\documentclass[11pt]{article}

\usepackage{amsmath,amssymb,amsthm}
\usepackage{amsfonts}
\usepackage{amscd}
\usepackage{amsbsy}
\usepackage{epsfig}
\usepackage{tikz}
\usepackage{tikz-cd}
\usepackage{url}
\usepackage{here}
\usepackage{enumitem}
\usepackage{a4wide}
\usepackage{hyperref}
\newtheorem{theorem}{Theorem}[section]
\newtheorem{lemma}[theorem]{Lemma}
\newtheorem{proposition}[theorem]{Proposition}
\newtheorem{corollary}[theorem]{Corollary}

\hypersetup{pdfborderstyle={/S/S/W .4},linkbordercolor=blue!50!black}

\newcommand{\RR}{\mathbf R}
\newcommand{\PP}{\mathbf P}
\newcommand{\GG}{\mathbf G}

\newcommand{\Eff}{\overline{\operatorname{Eff}}}
\newcommand{\Lin}{\operatorname{Lin}}
\newcommand{\Nef}{\operatorname{Nef}}

\newcommand{\wt}{\widetilde}
\newcommand{\iso}{\cong}

\newcommand{\Bs}{\operatorname{Bs}}

\newcommand{\arrow}{\rightarrow}
\newcommand{\dual}{\ast}

\newcommand{\num}[1]{[#1]}

\usetikzlibrary{matrix.skeleton}

\tikzset{vertical bars/.default={2,...,\numexpr\pgfmatrixcurrentcolumn-1},
         vertical bars/.code={\xdef\pgftablecurrentvbars{#1}},
         horizontal bars/.default={},
         horizontal bars/.code={\xdef\pgftablecurrenthbars{#1}}}

\tikzstyle{table}=[inner sep=0pt,
                   matrix of math nodes,
                   nodes={inner xsep=5pt,inner ysep=0pt,anchor=base,scale=#1,minimum width=30pt},
                   column 1/.style={nodes={execute at begin node=\text{\vrule width 0pt height 10.5pt depth 3.5pt}}},
                   label skeleton,
                   style even rows on layer={background}{pattern=crosshatch,pattern color=purple,opacity=.2},
                   vertical bars,horizontal bars,
                   execute at begin matrix={\def~{\phantom{-}}},
                   execute at end matrix={\xdef\pgfmatrixcurrentname{\tikzmatrixname}},
                   append after command=\pgfextra{%
  \foreach \col in \pgftablecurrentvbars {%
    \draw[opacity=.4] (\pgfmatrixcurrentname-inter-column-\col.north) -- (\pgfmatrixcurrentname-inter-column-\col.south);
  };
  \foreach \row in \pgftablecurrenthbars {%
    \draw[opacity=.4] (\pgfmatrixcurrentname-inter-row-\row.west) -- (\pgfmatrixcurrentname-inter-row-\row.east);
  };
  \draw[double distance=1]  (\pgfmatrixcurrentname-inter-column-1.north) -- (\pgfmatrixcurrentname-inter-column-1.south);
  \draw[double distance=1]  (\pgfmatrixcurrentname-inter-row-1.west)     -- (\pgfmatrixcurrentname-inter-row-1.east);
  \draw[white,line width=1] (\pgfmatrixcurrentname-inter-column-1.north) -- (\pgfmatrixcurrentname-inter-column-1.south);
}]
\tikzset{table/.default=1}

\usetikzlibrary{patterns,shapes.geometric}

\newcommand{\xmark}{\text{\it\sffamily x}\mskip2mu}
\newcommand{\cmark}{\mskip4mu\checkmark}

\title{Effective cycles on some linear blowups of projective spaces}
\date{}
\author{N. Pintye, A. Prendergast-Smith}
\begin{document}
\maketitle

Cones of curves and divisors have played a central role in birational geometry since the groundbreaking work of Mori in the early 1980s. There are general results, such as the Cone Theorem, describing the structure of these cones, as well as numerous explicit calculations in cases of geometric interest.

More recently, there has been increased interest in cycles of intermediate dimensions. Debarre--Ein--Lazarsfeld--Voisin \cite{DELV} showed that in general, these cycles do not share the good properties of divisors or curves: in particular, numerical positivity need not imply geometric positivity for such cycles. Nevertheless, there has been significant progress in extending the theoretical understanding of such cycles, due to Fulger--Lehmann \cite{FL1, FL2}, Ottem \cite{Ott} and others. By contrast, the number of examples in which cones of effective cycles have been explicitly computed is relatively small. The most significant results to date were found by Coskun--Lesieutre--Ottem \cite{CLO}, who computed cones of cycles on blowups of projective spaces at sets of points.

In this paper, we compute cones of effective cycles on some varieties obtained by blowing up general sets of lines in projective space. These cones are more complicated to compute than those of point blowups in two ways: first, a hyperplane in projective space cannot contain many general lines, and so inductive techniques tend to be less useful; second, the coefficients of the intersection form on the blowup vary with dimension, making uniform statements more difficult to find. In spite of these difficulties we are able to compute cones in some interesting examples, which we now explain.

Blowing up a small number of lines in projective space gives a toric variety, so the cone of effective cycles is generated by torus-invariant subvarieties, hence linear subspaces. Our main results show that linear generation continues to hold when the number of lines is increased beyond the toric range: for example, the blowup of $\PP^4$ in more than 2 lines is no longer toric, but we show in Theorem \ref{theorem-2cyclesP4} that its cone of 2-cycles is still linearly generated when we blow up in 3 or 4 lines. Similarly, in Theorem \ref{theorem-2cyclesP5}, we show that the cones of 2-cycles is linearly generated when we blow up at most 5 lines in $\PP^5$, but the cone of 3-cycles fails to be linearly generated once we blow up 4 lines. Finally, in Section \ref{section-curves}, we complement these theorems with some results about linear generation of cones of curves and divisors.

Our results are summarised in the following tables. In each table, the entry in row $k$ and column $r$ shows whether the cone of effective $k$-cycles on the blowup of projective space of the relevant dimension in $r$ general lines is linearly generated (or if the answer is not known). Note that once linear generation fails for a blowup, it fails for all further blowups, so any entry to the right of the symbol {\it\sffamily x} in a given row is also not linearly generated.


\begin{figure}[ht]
\centering
Dimension $4$\par
\vspace{1.2ex}
\begin{tikzpicture}
\matrix[table] (T) {
\vrule width 0pt depth 5pt\lower 1.5pt\hbox{$k$}\mskip-4mu\setminus\mskip-4mu\raise 2pt\hbox{$r$}&\smash{\leq}\,2&3&4&5&6&7&8&9&10\\
1&\cmark&\cmark&\cmark&\cmark&\cmark&\cmark&?&?&\xmark\\
2&\cmark&\cmark&\cmark&\xmark&&&&&\\
3&\cmark&\cmark&\cmark&\xmark&&&&&\\
};
\node[below left=-2] at (T-cell-2-7.east)  {\tiny\ref{prop-curvesinP4}};
\node[below left=-2] at (T-cell-2-10.east) {\tiny\ref{prop-curvesinP4}};
\node[below left=-2] at (T-cell-3-4.east)  {\tiny\ref{theorem-2cyclesP4}};
\node[below left=-2] at (T-cell-3-5.east)  {\tiny\ref{theorem-2cyclesP4}};
\node[below left=-2] at (T-cell-4-4.east)  {\tiny\ref{prop-divisorsonP4}};
\node[below left=-2] at (T-cell-4-5.east)  {\tiny\ref{prop-divisorsonP4}};
\end{tikzpicture}
\end{figure}

\begin{figure}[ht]
\centering
Dimension $5$\par
\vspace{1.2ex}
\begin{tikzpicture}
\matrix[table] (T) {
\vrule width 0pt depth 5pt\lower 1.5pt\hbox{$k$}\mskip-4mu\setminus\mskip-4mu\raise 2pt\hbox{$r$}&\smash{\leq}\,3&4&5&6\\
1&\cmark&\cmark&\cmark&\xmark\\
2&\cmark&\cmark&\cmark&?\\
3&\cmark&\xmark&&\\
4&\cmark&\xmark&&\\
};
\node[below left=-2] at (T-cell-2-4.east) {\tiny\ref{prop-curvesinP5}};
\node[below left=-2] at (T-cell-2-5.east) {\tiny\ref{prop-curvesinP5}};
\node[below left=-2] at (T-cell-3-4.east) {\tiny\ref{theorem-2cyclesP5}};
\node[below left=-2] at (T-cell-4-2.east) {\tiny\ref{prop-toriccones}};
\node[below left=-2] at (T-cell-4-3.east) {\tiny\ref{prop-3cyclesonP5}};
\node[below left=-2] at (T-cell-5-2.east) {\tiny\ref{prop-toriccones}};
\node[below left=-2] at (T-cell-5-3.east) {\tiny\ref{prop-divisorsonP5}};
\end{tikzpicture}
\end{figure}


The pattern we find agrees with Coskun--Lesieutre--Ottem's results, namely that as we blow up more, cones of lower-dimensional cycles remain linearly generated for longer than cones of higher-dimensional cycles. It would be interesting to find uniform bounds ensuring linear generation for blowups of projective space in general sets of linear subspaces of arbitrary dimension.

Thanks to Izzet Coskun and Elisa Postinghel for helpful conversations.
\section{Preliminaries}

We work throughout over an algebraically closed field of characteristic zero.

\subsection{Intersection theory} \label{subsection-int}

Our goal in this paper is to compute cones of cycles. The natural contexts for these cones are the spaces of numerical classes of cycles, which we now introduce. In the examples we will consider, these spaces are just Chow groups with real coefficients, but we use the language of numerical classes for consistency with the general theory. 

Let $X$ be a smooth proper variety of dimension $n$. Let $Z_k(X)$ denote the group of algebraic cycles of dimension $k$ on $X$. We define the vector space of numerical classes of $k$-cycles to be 
\begin{align*}N_k(X) := \left( Z_k(X)/\equiv \right) \otimes \RR\end{align*} where $\equiv$ denotes numerical equivalence of cycles. For each $k$, this is a finite-dimensional real vector space, and intersection gives a perfect pairing $N_k(X)\times N_{n-k}(X) \arrow \RR$. For convenience, we often write $N^k(X)$ instead of $N_{n-k}(X)$. For a $k$-dimensional subvariety $Z$ in $X$, we write $\num{Z}$ to denote its class in $N_k(X)$. A fundamental feature of this product is positivity of proper intersections: if $X$ is a smooth proper variety of dimension $n$, and $V$ and $W$ are subvarieties of dimension $k$ and $n-k$, respectively, such that $V \cap W$ is a finite set, then $\num{V} \cdot \num{W} \geq 0$.

A class $\alpha \in N_k(X)$ is {\it effective} if there are subvarieties $Z_1,\ldots,Z_m$ and non-negative real numbers $r_1,\ldots,r_m$ such that $\alpha = \sum_{i=1}^m r_i [Z_i]$. A class $\alpha \in N^k(X)$ is called \emph{nef} if $\alpha \cdot \num{Z} \geq 0$ for every $k$-dimensional subvariety $Z$ in $X$ or, equivalently, if $\alpha \cdot \beta \geq 0$ for every effective class $\beta \in N_k(X)$. We need some basic facts about the behaviour of nef cycles under morphisms:
\begin{proposition} \label{prop-pushpull}
  Let $f: Y \arrow X$ be a morphism of smooth projective varieties.
  \begin{enumerate}
  \item[(a)] If $\alpha \in N^k(X)$ is nef, then $f^*\alpha \in N^k(Y)$ is nef. 
    \item[(b)] If $f$ is surjective and $\alpha \in N_k(X)$ is a cycle such that $f^*\alpha$ is nef, then $\alpha$ is nef.
  \end{enumerate}
\end{proposition}
\begin{proof} (a): If $\beta \in N_k(Y)$ is effective, then $f_*\beta$ is also effective by definition of pushforward. So if $\alpha \in N^k(X)$ is nef, then using the projection formula for cycles, we get $f^*\alpha \cdot \beta = \alpha \cdot f_*\beta \geq 0$ for every effective cycle $\beta$ in $N_k(Y)$.

(b): Let $\beta \in N^k(X)$ be an effective class. Since $f$ is surjective, by a standard hyperplane section argument there exists an effective class $\widetilde{\beta} \in N^k(Y)$ such that $f_* \widetilde{\beta} = \beta$. By the projection formula and nefness of $f^*\alpha$, we have $\alpha \cdot \beta = f^* \alpha \cdot \widetilde{\beta} \geq 0$, showing that $\alpha$ is nef as required.
\end{proof}
In general the intersection of nef cycles need not be nef \cite[Corollary 2.2]{DELV}, but for divisors this is true: 
\begin{lemma}  \label{lemma-nefintersections}
Let $X$ be a smooth projective variety. If $D$ and $E$ are nef divisor classes on $X$, then $DE$ is a nef class in $N^2(X)$.
\end{lemma}
\begin{proof}
We need to prove that for any effective class $\alpha \in N_2(X)$, we have $DE \cdot \alpha \geq 0$. Since $E$ is nef, we can find a sequence of ample divisor classes $\{E_i\}$ converging to $E$ in $N^1(X)$. For each $E_i$, the intersection $E_i \alpha$ is effective, and so $D \cdot (E_i \, \alpha) \geq 0$. Taking the limit, we get $DE \cdot \alpha = \lim_i D \cdot (E_i \, \alpha) \geq 0$, as required.
\end{proof}

\subsection*{Numerical classes on blowups}
In the rest of the paper, we will write $X^n_{r,s}$ to denote the blowup of $\PP^n$ in a collection of $r$ general lines $L_1,\ldots,L_r$ and $s$ general points $p_1,\ldots,p_s$. Our main examples have $s=0$, and we denote these simply by $X^n_r$.

The ring $N^*(X^n_{r,s})=CH(X^n_{r,s}) \otimes \RR$ is generated by classes $H$, $E_i$ for $i=1,\ldots, r$, and $e_i$ for $i=1,\ldots,s$, which are respectively the pullback of the hyperplane class on $\PP^n$, the exceptional divisors of the blowups of the $L_i$ and the exceptional divisors of the blowups of the $p_i$. We will use the following intersection numbers among these classes \cite[Corollary 9.12]{3264}:
\begin{align*}
  H^n=1, \quad &E_i^n=(-1)^n(n-1), \quad e_i^n = (-1)^{n-1} \\
  H \cdot E_i^{n-1} = (-1)^n, \quad &H^j \cdot E_i^{n-j} = H^k \cdot e_i^{n-k}=0 \text{ for } i \geq 1, \, k \geq 0.
\end{align*}
We also need to know the numerical classes on $X^n_{r,s}$ of the proper transforms of certain subvarieties of $\PP^n$. The blowup formula \cite[Theorem 6.7]{Fulton} allows us to calculate these as long as we know the Segre classes of the blowup centre inside the subvariety: in particular, when we blow up lines and points, these are easy to compute. In particular, we note the following:
\begin{corollary} \label{corollary-classes}
 Let $X^n_{r,s}$ be the blowup of $\PP^n$ in $r$ general lines and $s$ general points. Let
  \begin{itemize}
  \item $U$ be a linear space of codimension $k$ intersecting $L_i$ transversely,
  \item $V$ be a linear space of codimension $k$ containing $L_i$, and let
  \item $Q$ be a quadric of codimension $k$ containing $L_i$.
  \end{itemize}
  The numerical classes of the proper transforms of these spaces have the following coefficients:
\begin{equation*}
\begin{tikzpicture}
\matrix[table] {%
  {} & H^k & HE_i^{k-1} & E_i^k \\
  \vrule width 0pt height 12pt depth 0pt\num{\widetilde{U}} & 1 & {(-1)}^{k-1}\phantom{k} & 0 \\
  \vrule width 0pt height 12pt depth 0pt\num{\widetilde{V}} & 1 & {(-1)}^{k-1}k & {(-1)}^k \\
  \vrule width 0pt height 12pt depth 0pt\num{\widetilde{Q}} & 2 & {(-1)}^{k-1}(k+1) & {(-1)}^k \\
};
\end{tikzpicture}
\end{equation*}  
If $Z$ is any subvariety of codimension $k$ containing $p_i$ as a smooth point, then the coefficient of $e_i^k$ in $[\wt{Z}]$ equals $(-1)^k$.
\end{corollary}

\subsection{Cones of cycles} \label{sect-cones}
For a smooth projective variety $X$, the \emph{pseudoeffective cone} $\Eff_k(X)$ is the closed convex cone in $N_k(X)$ generated by numerical classes of $k$-dimensional subvarieties of $X$. The \emph{nef cone} $\Nef^k(X)$ is the cone spanned by all nef classes in $N^k(X)$: in other words, it is the dual cone of $\Eff_k(X)$. 

Now we specialise the discussion to our examples $X^n_{r,s}$. A subvariety of $X^n_{r,s}$ is called {\it linear} if it is one of the following:
\begin{enumerate}
\item[(a)] the proper transform on $X^n_{r,s}$ of a linear subspace of $\PP^n$, or
\item[(b)] the pullback to $E_i \iso \PP^1 \times \PP^{n-2}$ of a linear subspace in one of the factors, or
\item[(c)] a linear subspace in $e_i \iso \PP^{n-1}$.
\end{enumerate}
The \emph{linear cone} $\Lin_k(X^n_{r,s})$ is the cone in $N_k(X^n_{r,s})$ generated by the finitely many classes of $k$-dimensional linear subvarieties. We say that the pseudoeffective cone of $k$-cycles on $\Eff(X^n_{r,s})$ is \emph{linearly generated} if it equals the linear cone $\Lin_k(X^n_{r,s})$. Note that any blowup map $X^n_{r,s} \arrow X^n_{r-a,s-b}$ maps the effective cone onto the effective cone and the linear cone onto the linear cone, so if $\Eff(X^n_{r,s})$ is linearly generated, then so too is $\Eff(X^n_{r-a,s-b})$.

\subsection{Toric varieties}
Cones of cycles on toric varieties are well-understood. For later use, let us record the facts we need:
\begin{proposition} \label{prop-toriccones}
Let $X$ be a normal proper toric variety. Then, $\Eff_k(X)$ is generated by the finitely many classes of $k$-dimensional torus-invariant subvarieties on $X$. Consequently, if the variety $X^n_{r,s}$ is toric, then $\Eff_k(X^n_{r,s})$ is linearly generated for all $k$. 
\end{proposition}
\begin{proof}
  The first statement is well-known; a reference is \cite[Proposition 3.1]{Li}. 

  For the second statement, note that the torus-invariant subvarieties of $\PP^n$ are exactly the coordinate subspaces, so if $X^n_{r,s}$ is a toric blowup of $\PP^n$, any torus-invariant subvariety on $X^n_{r,s}$ that comes from $\PP^n$ is the proper transform of a coordinate subspace and hence is linear. On the other hand, every exceptional divisor of $X \arrow \PP^n$ is of the form $E_i \iso \PP^1 \times \PP^{n-2}$ or $e_i \iso \PP^{n-1}$, so the torus-invariant subvarieties of the exceptional divisor are also linear.
\end{proof}

\subsection{Computations} \label{sect-comp}
In this paper, we will use computer algebra in several different contexts. In all cases, we use the computer algebra system Macaulay2. In particular, for all computations of dual numerical cones, we use the package {\tt Normaliz} \cite{Nor} for Macaulay2. Note that for compactness, we always list the generators of all cones ``up to permutation'': that is, a full list of generators is obtained from our list by permuting indices in the appropriate way.

The full outputs of our computations are available in ancillary files provided with this paper \cite{M2}. The name of each file in the repository indicates the result in the paper in which the output of the computation is used.

\section{Codimension 2 linear spaces} 
In this section, we prove that codimension 2 linear spaces incident to lines give nef classes in $X^n_r$ for $r\leq n \leq 5$. The main idea of the proof is to verify by a dimension count that we can find such a linear space properly intersecting any given subvariety of complementary dimension. As mentioned in the introduction, proper intersections are non-negative, so this is sufficient to prove our claim.

We begin with some preparatory results about intersections of Schubert cycles.
\begin{lemma} \label{lemma-linesinP3}
  Let $l_1,\ldots,l_4$ be a set of 4 distinct lines in $\PP^3$, and let $\Lambda \subset \GG(1,3)$ be the set of lines touching all 4. Then, one of the following is true:
  \begin{enumerate}[label=(\alph*)]
  \item the set $\Lambda$ has dimension 2, in which case one of the following is true:
    \begin{enumerate}
    \item[(i)] all 4 lines are concurrent, or
      \item[(ii)] all 4 lines are coplanar;
    \end{enumerate}
       \item the set $\Lambda$ has dimension 1, in which case one of the following is true:
    \begin{enumerate}
  \item[(i)] the lines are all pairwise skew and lie on a smooth quadric surface $Q \in \PP^3$, or
  \item[(ii)] there are exactly 2 pairs of intersecting lines, say $l_1,l_2$ and $l_3,l_4$, and the intersection point of $l_1,l_2$ lies in the plane spanned by $l_3,l_4$, or
  \item[(iii)] there are 3 concurrent lines, say $l_1, l_2, l_3$, and the line $l_4$ is skew to all others, or
    \item[(iv)] there are 3 coplanar lines, say $l_1, l_2, l_3$, and the line $l_4$ is skew to all the others;
    \end{enumerate}
  \item the set $\Lambda$ has dimension 0.
  \end{enumerate}
\end{lemma}
\begin{proof}
  For each case listed in $(a)$ and $(b)$ above, the given dimension count is straightforward to verify. It remains to check that in all other cases, the set $\Lambda$ has dimension 0.
  In the case that all lines are pairwise skew, this is well-known, so we must consider the cases in which some of the lines intersect. There are two possibilities not covered by the list above:
  \begin{itemize}
  \item two lines, say $l_1$ and $l_2$, intersect, and all other pairs are skew;
    \item there are exactly 2 pairs $l_1, \ l_2$ and $l_3, \, l_4$ of intersecting lines, and neither of the intersection points of the two pairs lies in the plane spanned by the other pair.
  \end{itemize}
  In the first case, any line intersecting all 4 lines must either lie in the plane spanned by $l_1$ and $l_2$, or pass through the intersection point of $l_1$ and $l_2$. In each case, however, there is a unique such line which also intersects $l_3$ and $l_4$.

  In the second case, no line contained in either of the planes spanned by two intersecting lines can intersect the other two lines. So the only line intersecting all 4 lines is the line joining the two intersection points of the pairs $l_1,l_2$ and $l_3,l_4$.\end{proof}

\begin{lemma} \label{lemma-nopoint}
  Let $l_1,\ldots,l_n$ be a set of $n$ general lines in $\PP^n$ for $n=4$ or $5$. Let $\Lambda \subset \GG(n-2,n)$ be the subset of the Grassmannian parametrising codimension-2 linear spaces touching all the lines. Then
\begin{enumerate}
  \item[(a)] $\Lambda$ is irreducible;
    \item[(b)] The intersection of all the linear spaces parametrised by points of $\Lambda$ is empty.
  \end{enumerate}
\end{lemma}
The restriction on $n$ can be removed at the cost of a more complicated proof, but the statement above is sufficient for our applications in later sections. The word ``general'' in the statement of the lemma means that the proof works for a Zariski open subset of points in the space of sets of $n$ lines; however, the proof does not produce such an open subset explicitly.
\begin{proof}
  (a): Let $U \subset \GG(1,n)^n$ be the open subset parametrising sets of $n$ distinct lines. Let $I \subset U \times \GG(n-2,n)$ be the incidence correspondence consisting of pairs $((L_1,\ldots,L_n),L)$ where $L$ is a codimension 2 linear space intersecting all of the $L_i$. Let $f: I \arrow U$ be the projection. We want to prove that a general fibre of $f$ is irreducible.

  Shrinking $U$ if necessary, we can assume that $f$ is flat; then by \cite[Theorem 12.2.1 (x)]{EGA}, the locus of integral fibres of $f$ is open. One checks (for example using Macaulay2) that for a particular point $u \in U$, the fibre $f^{-1}(u)$ is smooth and connected, hence integral, and so the general fibre of $f$ is integral and, in particular, irreducible. 

  (b): Suppose there is a point $p \in \PP^n$ such that every linear space parametrised by $\Lambda$ passes through $p$. Let $\Sigma_p \subset \GG(n-2,n)$ be the Schubert cycle parametrising linear spaces passing through $p$. In particular, we should have $\Lambda \subset \Sigma_p$. Let us show that this containment is impossible.

  First, note that $\Lambda$ has codimension $n$ in $\GG(n-2,n)$, while for any $p$, the Schubert variety $\Sigma_p$ has codimension 2. Considering the Pl\"ucker embedding of the Grassmannian $\GG(n-2,n)$ in projective space, we can view $\Lambda$ as $\GG(n-2,n) \cap H_1 \cap \cdots \cap H_n$ for certain hyperplanes $H_i$. Therefore, if $\Lambda \subset \Sigma_p$, we must have $\Lambda \subset \Sigma_p \cap H_1 \cap \cdots \cap H_{n-2}$.

  If the intersection $\Sigma_p \cap H_1 \cap \cdots \cap H_{n-2}$ is of the maximal codimension $n$, then $\Lambda$ must be an irreducible component of  $\Sigma_p \cap H_1 \cap \cdots \cap H_{n-2}$. However, the degree of  $\Sigma_p \cap H_1 \cap \cdots \cap H_{n-2}$ is the same as the degree of $\Sigma_p$, and Schubert calculus shows that this is strictly less than the degree of $\Lambda$, a contradiction.

  In general, suppose that $\Sigma_p \cap H_1 \cap \cdots \cap H_{n-2}$ is not of the maximal codimension $n$. We claim that we can move the hyperplanes $H_i$ to new hyperplanes $H_i^\prime$ such that both of the following hold:
    \begin{itemize}
    \item $\Sigma_p \cap H_1^\prime \cap \cdots \cap H_{n-2}^\prime$
      is of codimension $n$;
      \item $H_1^\prime \cap \cdots \cap H_{n-2}^\prime = H_1 \cap \cdots \cdots H_{n-2}.$
    \end{itemize}
    (Note that the new hyperplanes $H_i^\prime$ in general no longer correspond to Schubert varieties in the Grassmannian $\GG(n-2,n)$, but that does not affect our proof.) Given the claim, we can then write $\Lambda$ as $\GG(n-2,n) \cap H_1^\prime \cap \cdots \cap H_{n-2}^\prime$, and the argument from the previous paragraph applies again to complete the proof.

    It remains to prove the claim. Write $Z=H_1 \cap \cdots \cap H_{n-2}$. For $i<n-2$, assume we have chosen hyperplanes $H_1^\prime,\ldots,H_i^\prime$ such that each of them contains $Z$, and $\Sigma_p \cap H_1^\prime \cdots \cap H_i^\prime$ has codimension $i+2$. Since $i+2<n$, we see that $\Sigma_p \cap H_1^\prime \cdots \cap H_i^\prime$ is not contained in $\Lambda$, and since $\Lambda = \GG(n-2,n) \cap Z$, this proves it is not contained in $Z$ either. So we can find another hyperplane $H_{i+1}^\prime$ which contains $Z$ but does not contain $\Sigma_p \cap H_1^\prime \cdots \cap H_i^\prime$. Hence, the intersection $\Sigma_p \cap H_1^\prime \cdots \cap H_i^\prime \cap H_{i+1}^\prime$ has codimension $i+3$. Continuing in this way, we end up with hyperplanes $H_1^\prime, \ldots, H_{n-2}^\prime$ satisfying the two conditions above, as required.
\end{proof}

Now we can prove our first main result about nefness of codimension 2 linear spaces in $\PP^4$. The idea is to project away from a point and use the information from the previous lemmas about configurations of 4 lines in $\PP^3$.
\begin{theorem}\label{theorem-planesinP4}
Let $r \leq 4$. Let $L^4_r$ be the proper transform on $X^4_r$ of a codimension 2 linear space in $\PP^4$ that intersects all the blown-up lines properly. Then, $L^4_r$ is nef.
  \end{theorem}
\begin{proof}
  We first observe that if $L^4_r$ is nef on $X^4_r$, then $L^n_{r-1}$ is nef on $X^4_{r-1}$. To see this, note that the pullback of the class $\num{L^4_{r-1}}$ equals $\num{L^4_r}+\num{F}$, where $F$ is a fibre of the blowup. If $L^4_r$ is nef, then any irreducible surface that has negative intersection with the pullback of $L^4_{r-1}$ must have negative intersection with $F$ and so must be contained in $E_r$, since $F$ is a nef divisor in $E_r$. But surfaces contained in $E_r$ are contracted by the blowup map, so they have zero intersection with the pullback of $\num{L^4_{r-1}}$ by the projection formula. So the pullback of $\num{L^4_{r-1}}$ is nef, and therefore $L^4_{r-1}$ is nef by Proposition \ref{prop-pushpull}. So it suffices to prove that $L^4_4$ is nef.

The restricition of $L_4^4$ to any of the divisors $E_i$ is an effective curve class, hence nef, so if $S$ is an irreducible surface contained inside one of the divisors $E_i$, then $L^4_4 \cdot \tilde{S} \geq 0$. We can therefore restrict our attention to irreducible surfaces $\tilde{S}$ that are proper transforms of surfaces $S$ in $\PP^4$. For such a surface, the intersection $\tilde{S} \cap E_i$ is 1-dimensional, hence a union of curves. We can write it in the form $\tilde{S} \cap E_i = C_1 \cup \cdots \cup C_k \cup \Gamma_1 \cup \cdots \cup \Gamma_j$, where the $C_i$ are curves contained in fibres of the blowdown map $X^4_4 \arrow \PP^4$, and the $\Gamma_j$ intersect each fibre of $\pi$ in finitely many points. By Lemma \ref{lemma-nopoint}, we can choose a plane $L^4_4$ that is disjoint from any given finite set of fibres of $\pi$, and for such a plane, we get $L^4_4 \cap \tilde{S} \cap E_i = L^4_4 \cap ( \cup_{k=1}^j \Gamma_j)$, a finite set of points. Therefore, if $\tilde{S}$ is a surface intersecting every plane $L^4_4$ non-properly, we see that $S \cap L$ has dimension at least 1 for every plane $L \subset \PP^4$ intersecting all 4 lines.

So suppose that $S \subset \PP^4$ is an irreducible surface such that $\dim (S \cap L ) \geq 1$ for every plane $L \subset \PP^4$ intersecting all 4 lines. We form the following incidence correspondence:
\begin{center}
  \begin{tikzcd}
I = \left\{ (L, p) \mid L \in \Lambda, \, p \in S \cap L \right\}  \arrow{r}{\pi_1} \arrow{d}{\pi_2} & S \\
\Lambda &
  \end{tikzcd}
\end{center}
Here $\Lambda \subset \GG(2,4)$ is the subset of the Grassmannian parametrising planes intersecting all 4 lines. By Lemma \ref{lemma-nopoint}, $\Lambda$ is irreducible of dimension 2. Hence, by our assumption on the dimension of the fibres of $\pi_2$, we see that $I$ has dimension at least 3.

Every fibre $\pi_1^{-1}(p)$ is a subset of $\Lambda$, which is irreducible of dimension 2, so any fibre of dimension 2 must equal $\Lambda$. But if a fibre $\pi_1^{-1}(p)$ equals $\Lambda$, then all the planes parametrised by $\Lambda$ pass through the point $p$, contradicting Lemma \ref{lemma-nopoint} (b). Hence, no fibre $\pi_1^{-1}(p)$ has dimension 2.

Therefore, every fibre of $\pi_1$ has dimension 1, and so $\pi_1$ is surjective. That is, for every point $p \in S$, there are infinitely many planes $L$ passing through $p$ and intersecting all 4 lines. We will show that this is impossible.

By Lemma \ref{lemma-appendix1}, we may assume that $S$ is not contained in any of the linear spaces $\operatorname{Span}(L_i,L_j)$. By this assumption, if $p \in S$ is a general point, then when we project away from $p$, the images of our lines $L_1,\ldots,L_4$ give 4 skew lines $l_1, \ldots, l_4$ in $\PP^3$. Under this projection, planes $L \subset \PP^4$ passing through $p$ and intersecting all the lines $L_i$ correspond to lines $l \subset \PP^3$ intersecting all the lines $l_i$. So if there are infinitely many planes $L$ passing through $p$ and intersecting all 4 lines, then there must be infinitely many lines in $\PP^3$ intersecting the 4 skew lines $l_1,\ldots,l_4$.

For 4 skew lines $l_1,\ldots,l_4$ in $\PP^3$, there are at most 2 lines intersecting them all unless the 4 lines all lie on a quadric $Q \subset \PP^3$. So we must have that $p$ is contained in the vertex of a quadric cone $Q' \subset \PP^4$, which also contains the lines $L_1,\ldots,L_4$. Let us examine the possibilities for the rank of $Q'$:
    \begin{itemize}
    \item rank 1: in this case, all the lines $L_i$ would be contained in a hyperplane, contradicting generality;
    \item rank 2: in this case, all the lines would be contained in a union of 2 hyperplanes whose intersection contains $p$. Each of the 2 hyperplanes would be spanned by 2 of the lines $L_i$, contradicting the assumption that $p$ is not contained in the span of any 2 of the $L_i$;
      \item rank 3: in this case the vertex of $Q'$ is a line $L$, and projecting from $L$, maps $Q'$ to a smooth conic $Q'' \subset \PP^2$. On the other hand, any line in $Q'$ which is disjoint from $L$ would map to a line in $\PP^2$ contained in $Q''$, which is impossible. So all lines in $Q'$, in particular all the $L_i$, must intersect a fixed line $L$. Again by generality this is impossible.
    \end{itemize}
    We conclude that any such quadric $Q'$ must have rank 4, hence its vertex has dimension 0.

    The linear system $V$ of quadrics containing the $L_i$ has dimension 2. In order to complete the proof, we now analyse 2 possible cases.

    If the general member of $V$ is smooth, then the subset of singular quadrics has dimension at most 1. We just proved that, except for the 3 quadrics of rank 2 which are unions of hyperplanes $\operatorname{Span}(L_i,L_j)$, the vertex of any such quadric has dimension 0. So we get a 1-dimensional set of vertices of quadrics outside the subsets $\operatorname{Span}(L_i,L_j)$. This 1-dimensional set cannot contain any surface $S$, so there cannot exist a surface $S$ outside the subspaces $\operatorname{Span}(L_i,L_j)$ such that through each point of $S$ there pass infinitely many planes touching all the lines $L_i$.

    If the general member of $V$ is singular, then Bertini's theorem still guarantees that the set of singularities of a general member of $V$ is contained in the base locus $\Bs (V)$. Other than the 3 rank-2 quadrics from the last paragraph, the set of members of $V$ whose singular set is not contained in $\Bs (V)$ is at most 1-dimensional, so the set of singular points of such quadrics again gives a 1-dimensional set. On the other hand, $\Bs (V)$ is also 1-dimensional, as one sees, for example, by intersecting the 3 rank-2 quadrics, so we get a 1-dimensional set of vertices altogether. Again, this set cannot contain a surface $S$.\end{proof}

Next we prove the corresponding result for codimension 2 linear spaces in $\PP^5$. The idea of the proof in this case is to project away from a line, rather than a point, and then argue as before.
\begin{theorem} \label{theorem-planesinP5}
Let $r \leq 5$. Let $L^5_r$ be the proper transform on $X^5_r$ of a codimension 2 linear space in $\PP^5$ that intersects all the blown-up lines properly. Then, $L^5_r$ is nef.
\end{theorem}
\begin{proof}
  As in the previous theorem, it suffices to prove the result when $r=5$. We suppose for contradiction that there is an irreducible surface $S \subset \PP^5$ such that $\operatorname{dim}(S \cap L) \geq 1$ for every codimension 2 linear space $L$ that intersects all 5 lines. Again, we form the incidence correspondence
  \begin{center}
  \begin{tikzcd}
I = \left\{ (L, p) \mid L \in \Lambda, \, p \in S \cap L \right\}  \arrow{r}{\pi_1} \arrow{d}{\pi_2} & S \\
\Lambda &
  \end{tikzcd}
\end{center}
  where now $\Lambda \subset \GG(3,5)$ is the subset of the Grassmannian parametrising linear spaces intersecting all 5 lines. Arguing exactly as before, we see that all fibres of $\pi_1$ must have dimension 2. We will show that the locus of points $p \in \PP^5$ through which we have a 2-dimensional family of linear spaces from $\Lambda$ does not contain any irreducible surfaces except for those contained in subspaces $\operatorname{Span}(L_i,L_J)$. As the proper transform of such a subspace is a toric variety, its cone of surfaces is linearly generated, and so $L^5_5$ has a non-negative intersection product with the class of any such surface.

  So assume $p \in \PP^5$ is a point such that the set $\Lambda_p$ of linear spaces in $\Lambda$ that pass through $p$ is 2-dimensional. By Proposition \ref{prop-toriccones}, we can assume the surface $S$ above does not lie in one of the linear spaces $\operatorname{Span}(L_i,L_j)$, so it is enough to consider points $p$ not in any of these linear spaces.

  Fix one of the lines, say $L_1$. First, we claim that for any point $q \in L_1$, the subset $\Lambda_{pq} \subset \Lambda_p$ consisting of linear spaces through both $p$ and $q$ has dimension 1. If this were not the case, there would be a point $q \in L_1$ such that the family of linear spaces through $p$ and $q$ has dimension 2. Projecting away from the line joining $p$ and $q$, the lines $L_2, L_3, L_4, L_5$ would then map to lines $l_2, l_3, l_4, l_5$ in $\PP^3$ with a 2-dimensional family of lines intersecting all 4. This can only happen in the following cases: first, two of the lines coincide; second, all four lines pass through a common point $p \in \PP^3$; third, all four lines lie in a common plane $P \subset \PP^3$. The first case only occurs if the centre of projection $\overline{pq}$ is contained in $\operatorname{Span}(L_i,L_j)$ for some $i \neq j$, but since $q$ is a point in $L_1$, this means that $L_1$ intersects $\operatorname{Span}(L_i,L_j)$, contradicting generality of the lines. The second case occurs only if there is a 2-dimensional linear space in $\PP^5$ (namely, the cone over the point $p$) intersecting all 5 lines $L_i$, and again, this contradicts generality. The third cases only occurs if there is a hyperplane in $\PP^5$ (namely, the cone over $P$) containing all 5 lines $L_i$, and again, this contradicts generality.
  
  So we see that for any $q \in L_1$, the set of linear spaces through $p$ and $q$ and intersecting the lines $L_2,L_3,L_4,L_5$ has dimension 1. We may assume that the line $\overline{pq}$ is not contained in any of the linear subspaces $\operatorname{Span}(L_i,L_j)$, so projecting away from the line $\overline{pq}$. we obtain a set of 4 distinct lines in $\PP^3$ such that the family of lines in $\PP^3$ touching all 4 has dimension 1. According to Lemma \ref{lemma-linesinP3}, either two of the lines intersect or else they are pairwise skew and lie on a smooth quadric in $\PP^3$.

  Let us first deal with the case when two of the lines intersect. We will think of projection away from the line $\overline{pq}$ as projection away from $p$ first, followed by projection away from the image of $q$ in $\PP^4$. As explained above, we can assume that $p$ does not lie in any of the linear spaces $\operatorname{Span}(L_i,L_j)$, so first projecting away from $p$ gives 5 skew lines $l_1,\ldots,l_5$ in $\PP^4$. We next project away from a point $\tilde{q}$ on $l_1$. If $l_1$ is contained in any of the hyperplanes $\operatorname{Span}(l_i,l_j)$, then in $\PP^5$, we would have three lines $L_1, \, L_i, \, L_j$ contained in a hyperplane, contradicting generality. So $l_1$ meets each of the hyperplanes $\operatorname{Span}(l_i,l_j)$ in a single point. Choosing $\tilde{q}$ to be different from all of these points, the projection away from $q$ then gives us 4 pairwise skew lines in $\PP^3$. 
  
So we may suppose that the 4 lines are pairwise skew and lie on a smooth quadric surface in $\PP^3$. By taking the cone over this quadric, we get a quadric in $\PP^5$ of corank 2 that contains $L_2, L_3, L_4, L_5$ and whose vertex is a line intersecting $L_1$ and passing through $p$. Moreover, for each $q \in L_1$, we get such a quadric, so there is a 1-dimensional family of lines through $p$ that are vertices of quadrics of this type. We will prove that the set of such points $p$ either has dimension at most 1 or is a plane in $\PP^5$.

  By Lemma \ref{lemma-transversality}, the family of quadrics in $\PP^5$ of corank 2 that contain the lines $L_2,\ldots,L_5$ and whose vertex intersects $L_1$ is of dimension 2. Call this 2-dimensional family $\mathcal F$ and consider the following incidence correspondence:

  \begin{center}
  \begin{tikzcd}
J = \left\{ (Q, p) \mid Q \in \mathcal F \text{ and } p \text{ lies on the vertex line of } Q \right\}  \arrow{r}{\pi_1} \arrow{d}{\pi_2} & \mathcal \PP^5 \\
\mathcal F&
  \end{tikzcd}
  \end{center}

 All fibres of $\pi_2$ are lines, so every irreducible component of $J$ has dimension 3. We may assume that $J$ is irreducible: if not, we apply the same argument to each component of $J$ in turn. We distinguish 2 possible cases. If $\pi_1$ is generically finite, then the points $p \in \PP^5$ which lie on a 1-dimensional family of vertex lines of members of $\mathcal F$ are contained in a proper closed subset $Z$ of $\pi_1(J)$. The preimage $\pi_1^{-1}(Z)$ is a proper closed subset of $J$, hence has dimension at most 2, and the fibres of $\pi_1$ over points of $Z$ are 1-dimensional by hypothesis. Hence, $Z$ has dimension at most 1. If $\pi_1$ is not generically finite, then $\pi_1(J)$ is irreducible of dimension at most 2. For each point $p \in \pi(J)$, there is a 1-dimensional family of vertex lines touching $L_1$ and passing through $p$. Such a family sweeps out a plane $\Pi$ inside $\PP^5$, and so $\pi_1(J)$ is a plane.
\end{proof}

\begin{lemma} \label{lemma-transversality}
  For any $k \in \{0,\ldots,n-1\}$ and any $N$, the set $\Lambda(k,N,n)$ of quadrics in $\PP^n$ of corank $k$ and containing $N$ general lines has the expected codimension
  \begin{align*}
 e(k,N,n) := \operatorname{max} \left\{ 3N + \binom{k+1}{2}, \binom{n+2}{2} \right\}.
  \end{align*}
  Moreover, for $k \geq 1$, the set $\Lambda_v(k,N,n)$ of those quadrics in $\Lambda(k,N,n)$ whose vertex intersects another general line has the expected codimension
   \begin{align*}
 \epsilon(k,N,n) :=  \operatorname{max} \left\{ e(k,N)+n-k-1,  \binom{n+2}{2} \right\}.
  \end{align*}
\end{lemma}
In particular, with $n=5$, $N=4$ and $k=2$, we see that the locus $\Lambda_v(2,4,5)$ of quadrics in $\PP^5$ of corank $2$ containing 4 general lines and with vertex intersecting another general line has dimension
\begin{align*}
  \binom{5+2}{2} -1 - \epsilon(2,4,5) & = 20-3 \cdot 4-\binom{3}{2}-3 = 2
\end{align*}
as claimed in the proof of Theorem \ref{theorem-planesinP5}.
\begin{proof}
  For $1\leq i \leq N$, let $\Lambda(L_i)$ denote the set of quadrics in $\PP^n$ that contain the $i$-th line $L_i$, and let $\lambda(L_i)$ denote the intersection of $\Lambda(L_i)$ with the set $R_k$ of quadrics of corank $k$. Then $\lambda(L_i)$ has codimension 3 in $R_k$. To see this, one can for example fix the vertex $l$ and project away $\pi_l: \PP^n \dashrightarrow \PP^{n-k}$: quadrics with vertex $l$ and containing $L_i$ then correspond to smooth quadrics in $\PP^{n-k}$ containing $\pi_l(L_i)$. If $L_i$ is disjoint from $l$, this clearly gives a set of codimension 3. Varying $l$ among all linear spaces disjoint from $L_i$, we then get a subset of codimension 3 in $R_k$. If $L_i$ intersects $l$, then $\pi_l(L_i)$ is a point, so we get one condition on the smooth quadrics; however, for $n \geq 4$, the condition for $l$ to intersect a fixed line imposes $n-2 \geq 2$ conditions, and so we get codimension at least 3 in this case too. 
  
  For any $k$, the group $PGL(n+1)$ acts transitively on $R_k$ and maps $\lambda(L_i)$ to $\lambda(L'_i)$ for some other line $L'_i$ in $\PP^5$. For each $i$, we can apply Kleiman's transversality theorem \cite{Kle} to each component of $\lambda(L_i)$ to find a Zariski-open subset of $PGL(n+1)$ that moves the component into proper position relative to $\bigcap_{1 \leq j<i} \lambda(L_j)$. Intersecting these open subsets, we get a nonempty subset of elements moving every component of $\lambda(L_i)$ into proper position relative to $\bigcap_{1 \leq j<i} \lambda(L_j)$, and therefore the intersection $\lambda(L'_i) \, \cap \, \bigcap_{j<i} \lambda(L'_j)$ has the expected codimension $\binom{k+1}{2} + 3i$. Putting $i=n$, we get the claimed codimension $e(k,N,n)$ of $\Lambda(k,N,n)$.

  To prove the claimed codimension $\epsilon(k,N,n)$ of $\Lambda_v(k,N,n)$, for a line $L$ we write $\lambda_v(L)$ to denote the set of quadrics in $R_k$ whose vertex intersects $L$. Then $\lambda_v(L)$ has codimension $n-k-1$ in $R_k$, as one sees again by projection away from the vertex. Then the same argument as in the previous paragraph applies again to show that the codimension of $\Lambda_v(k,N,n)$ in $\Lambda(k,N,n)$ is $n-k-1$.
\end{proof}

\section{\texorpdfstring{2-cycles on $X^4_r$ for $r \leq 4$}{2-cycles on X\^{}4\_r for r at most 4}}
In the next two sections, we will prove our main results about linear generation of cones of cycles. We begin with the case of lines in $\PP^4$. In this case, $N_2(X^4_r)=N^2(X^4_r)$ has a basis consisting of the classes
\begin{align*}
H^2, \quad F_i := HE_i, \quad G_i := \ -E_i^2 \quad (i=1,\ldots r)
\end{align*}
where we have chosen signs so that effective classes in the exceptional divisors have positive coefficients with respect to the basis. 

The intersections among these classes are given by the following matrix:
\begin{equation*}
\begin{tikzpicture}
\matrix[table] {%
  {}& H^2 & F_i & G_i\\
  H^2 & 1 & 0 & 0 \\
  F_j & 0 & 0 & -\delta_{ij} \\
  G_j & 0 & -\delta_{ij} & \mskip4mu 3\delta_{ij}\\
};
\end{tikzpicture}
\end{equation*}
Using Corollary \ref{corollary-classes}, we can write down all the classes of linear subvarieties in $X^4_4$. The linear cone $\Lin_2(X^4_4)$ is then generated by the following list of classes, in which (as explained in Section \ref{sect-comp}) we list generators up to permutations of indices:
\begin{equation*}
\def\m{\phantom{-}}%
\begin{tikzpicture} 
\matrix[table,column 1/.style={nodes={minimum width=0pt,execute at begin node=\text{\strut}}}] {%
  {}& H^2    &F_1  &F_2  & F_3 & F_4 & G_1 & G_2 & G_3 & G_4 \\
  {}& 0&\m1\;&\m0\;&\m0\;&\m0\;&\m0\;& 0 & 0 & 0\\
  {}& 0&\m1\;&\m0\;&\m0\;&\m0\;&\m1\;& 0 & 0 & 0\\
  {}& 1&\m0\;&\m0\;&\m0\;&\m0\;&\m0\;& 0 & 0 & 0\\
  {}& 1& -1\;& -1\;& -1\;& -1\;&\m0\;& 0 & 0 & 0\\
  {}& 1& -2\;& -1\;& -1\;&\m0\;& -1\;& 0 & 0 & 0\\
};
\end{tikzpicture}
\end{equation*}
Before stating our main result on linear generation, we record one fact that will save work when verifying that certain classes are nef. 
\begin{lemma} \label{lemma-maximality}
  Let $\alpha \in N^k(X^n_r)$ be a nef class. Let $\beta$ be any class of the form $\beta = \alpha + \sum_i \num{Z_i}$, where $\{Z_i\}$ are subvarieties of $X^n_r$ contained in exceptional divisors. If $\beta$ is contained in $\Lin^*_k(X^n_r)$, then $\beta$ is also nef.
\end{lemma}
\begin{proof}
 We must show that for every irreducible subvariety $S$ of dimension $k$ in $X^n_r$, we have $\beta \cdot \num{S} \geq 0$. 
 
  If $S \subset E_j$ for some $j$, then since $E_j$ is toric, we have $\num{S} \in \Lin_k(E_j) \subset \Lin_k(X^n_r)$, and hence by hypothesis, $\beta \cdot \num{S} \geq 0$. 
 
 If $S$ is not contained in any exceptional divisor $E_j$, then it intersects each $E_j$ either in the empty set or a in set of dimension $k-1$. If $S \cap E_j$ is non-empty and $Z_i \subset E_j$ is one of the subvarieties appearing in $\beta$, then we can compute $\num{S} \cdot \num{Z_i}$ as $(\num{S \cap E_j} \cdot \num{Z_i})_{E_j}$, where the subscript indicates that the intersection is considered in the ambient space $E_j$. Since $E_j \iso \PP^1 \times \PP^{n-2}$, the intersection of any two effective cycles is again effective, so $\num{S} \cdot \num{Z_i} = (\num{S \cap E_j} \cdot \num{Z_i})_{E_j} \geq 0$. Since $\alpha$ is nef, we conclude that $\beta \cdot \num{S} \geq 0$, as required.
 \end{proof}

\begin{theorem} \label{theorem-2cyclesP4}
The effective cone of 2-cycles $\Eff_2(X^4_r)$ is linearly generated if and only if  $r \leq 4$. 
\end{theorem}
\begin{proof}
As explained in Section \ref{sect-cones}, to prove linear generation it is enough to consider the case $r=4$. Our strategy is to use the list of linear classes above to compute generators for the dual of the linear cone $\Lin_2(X^4_4)^\dual$ and verify that the generators are indeed nef classes. The generators of $\Lin_2(X^4_4)^\dual$ are as follows:
\begin{equation*}
\def\m{\phantom{-}}%
\begin{tikzpicture}
\matrix[table] {%
  {}          & H^2 & F_1 & F_2 & F_3 & F_4 & G_1 & G_2 & G_3 & G_4\\
  \alpha      & 1   &\m0\;&\m0\;&\m0\;&\m0\;&\m0\;&\m0\;&\m0\;&\m0\;\\
  \beta       & 1   & -1\;&\m0\;&\m0\;&\m0\;&\m0\;&\m0\;&\m0\;&\m0\;\\
  \gamma      & 1   & -1\;& -1\;&\m0\;&\m0\;&\m0\;&\m0\;&\m0\;&\m0\;\\
  \delta      & 1   & -1\;& -1\;& -1\;&\m0\;&\m0\;&\m0\;&\m0\;&\m0\;\\
  \varepsilon & 1   & -1\;& -1\;& -1\;& -1\;&\m0\;&\m0\;&\m0\;&\m0\;\\
  \hline
  \pi         & 1   & -2\;&\m0\;&\m0\;&\m0\;& -1\;&\m0\;&\m0\;&\m0\;\\
  \hline
  \lambda     & 3   & -2\;& -2\;& -2\;& -1\;& -1\;& -1\;& -1\;&\m0\;\\
  \mu         & 4   & -3\;& -3\;& -2\;& -2\;& -1\;& -1\;& -1\;& -1\;\\
  \nu         & 4   & -3\;& -3\;& -3\;& -2\;& -1\;& -1\;& -1\;& -1\;\\
  \xi         & 4   & -3\;& -3\;& -3\;& -3\;& -1\;& -1\;& -1\;& -1\;\\
};
\end{tikzpicture}
\end{equation*}
The class $\varepsilon$ is represented by the proper transform of a 2-dimensional linear subspace in $\PP^4$ intersecting all 4 lines, hence it is nef by Theorem \ref{theorem-planesinP4}. Lemma \ref{lemma-maximality} then implies that the classes $\alpha$ to $\delta$ are also nef.

The class $\pi$ is pulled back from a class $\pi^\prime$ on the toric variety $X^4_1$. It is straightforward to check that $\pi^\prime$ is in the cone $\Lin_2^\dual (X^4_1)$, so is nef by Proposition \ref{prop-toriccones}, and therefore by Proposition \ref{prop-pushpull}, the class $\pi$ is nef too. 

It remains to deal with the classes $\lambda$ to $\xi$. Again, by Lemma \ref{lemma-maximality}, it is enough to show that $\lambda$ and $\xi$ are nef.

To show that $\lambda$ and $\xi$ are nef classes, we will decompose them into effective classes and analyse the summands geometrically. In each table below, the rows sum up to the class in the top-left corner. The symbol $\pi_{ij}$ denotes the class of the proper transform of a plane containing $L_i$ and intersecting $L_j$, while $\gamma_k$ denotes the class of the proper transform of a plane containing $L_i$. 
\begin{equation*}
\def\m{\phantom{-}}%
\begin{tikzpicture}
\matrix[table] {%
    \lambda  & H^2 & F_1 & F_2 & F_3 & F_4 & G_1 & G_2 & G_3 & G_4 \\
    \pi_{14} & 1   & -2\;&\m0\;&\m0\;& -1\;& -1\;&\m0\;&\m0\;& 0 \\
    \gamma_2 & 1   &\m0\;& -2\;&\m0\;&\m0\;&\m0\;& -1\;&\m0\;& 0 \\
    \gamma_3 & 1   &\m0\;&\m0\;& -2\;&\m0\;&\m0\;&\m0\;& -1\;& 0 \\
};
\end{tikzpicture}
\end{equation*}
\par
\begin{equation*}
\def\m{\phantom{-}}%
\begin{tikzpicture}
\matrix[table] {%
    \xi      & H^2 & F_1 & F_2 & F_3 & F_4 & G_1 & G_2 & G_3 & G_4 \\
    \pi_{12} & 1   & -2\;& -1\;&\m0\;&\m0\;& -1\;&\m0\;&\m0\;&\m0\;\\
    \pi_{21} & 1   & -1\;& -2\;&\m0\;&\m0\;&\m0\;& -1\;&\m0\;&\m0\;\\
    \pi_{34} & 1   &\m0\;&\m0\;& -2\;& -1\;&\m0\;&\m0\;& -1\;&\m0\;\\
    \pi_{43} & 1   &\m0\;&\m0\;& -1\;& -2\;&\m0\;&\m0\;&\m0\;& -1\;\\
};
\end{tikzpicture}
\end{equation*}
We have already noted that the classes $\gamma_2$ and $\gamma_3$ are nef. The classes $\pi_{ij}$ are not nef, but we will show that any surface intersecting a class $\pi_{ij}$ from the above tables negatively must nevertheless have non-negative intersection with $\lambda$ and $\xi$.

For convenience, let us consider $\pi_{12}$; other cases are identical. Let $H=\operatorname{Span}(L_1,L_2)$, and let $\widetilde{H}$ be the proper transform of $H$ on $X^4_4$. By generality, the lines $L_3$ and $L_4$ each intersect $H$ in a point, and so $\widetilde{H} \iso X^3_{2,2}$. Now, $\pi_{12}$ is a divisor inside $\widetilde{H}$, and by Lemma \ref{lemma-appendix1}, it is nef. Therefore if $Z \subset X^4_4$ is an irreducible surface that is not contained in $\widetilde{H}$, we have $Z \cdot \pi_{12} >0$. On the other hand, if $Z$ is contained in $\widetilde{H}$, then we know that $Z$ is linear by Lemma \ref{lemma-appendix1}. Since $\lambda$ and $\xi$ are both in the dual of the linear cone $\operatorname{Lin}_2(X^4_4)$, they must both have non-negative intersection with $Z$.

Finally, to prove that linear generation does not hold for $r \geq 5$, it is enough to consider the case $r=5$. Choose any linear subspace spanned by two of the lines, say $H=\operatorname{Span}(L_1,L_2)$. The other 3 lines intersect $H$ in 3 points $p_3, p_4, p_5$. Counting dimensions, there is a quadric surface $Q$ inside $H$ containing the lines $L_1$ and $L_2$ and the points $p_3, p_4, p_5$. Blowing up, Corollary \ref{corollary-classes} tells us that the class of the proper transform of $Q$ on $X^4_5$ is
\begin{align*}
  \num{\tilde{Q}}&=2H^2-3F_1-3F_2-F_3-F_4-F_5-G_1-G_2
\end{align*}
  and it is straightforward to check that this is not in the linear cone $\Lin_2(X^4_5)$. 
\end{proof}

\section{\texorpdfstring{2-cycles on $X^5_r$ for $r \leq 5$}{2-cycles on X\^{}5\_r for r at most 5}}
The space $N^2(X^5_r)$ has a basis consisting of the classes
\begin{align*}
H^2, \quad F_i := H E_i \quad (i=1,\ldots,r), \quad G_i :=  -E_i^2 \quad (i=1, \ldots r)
\end{align*}
and the space $N_2(X^5_r)$ has a basis consisting of the classes
\begin{align*}
 H^3, \quad f_i := -HE_i^2 \quad (i=1,\ldots,r), \quad g_i := E_i^3 \quad (i=1, \ldots r)
\end{align*}
where, again, signs are chosen so that effective cycles in exceptional divisors have positive coefficients in the basis.

The intersections among these are as follows:
\begin{equation*}
\begin{tikzpicture}
\matrix[table] {%
  {} & H^2 & F_i & G_i\\
  H^3 & 1 & 0 & 0\\
  f_j & 0 & 0 & -\delta_{ij}\\
  g_j & 0 & -\delta_{ij} & 4\delta_{ij}\\
};
\end{tikzpicture}
\end{equation*}
The linear cone $\Lin_2(X^5_5)$ is then generated by the following classes:
\begin{equation*}
\def\m{\phantom{-}}%
\begin{tikzpicture}
\matrix[table,column 1/.style={nodes={minimum width=0pt,execute at begin node=\text{\strut}}}] {%
  {} & H^3 & f_1 & f_2 & f_3 & f_4 & f_5 & g_1 & g_2 & g_3 & g_4 & g_5\\
  {} & 0   &\m1\;&\m0\;&\m0\;&\m0\;& 0   &\m0\;& 0   & 0   & 0   & 0\\
  {} & 0   &\m1\;&\m0\;&\m0\;&\m0\;& 0   &\m1\;& 0   & 0   & 0   & 0\\
  {} & 1   &\m0\;&\m0\;&\m0\;&\m0\;& 0   &\m0\;& 0   & 0   & 0   & 0\\
  {} & 1   & -1\;& -1\;& -1\;& -1\;& 0   &\m0\;& 0   & 0   & 0   & 0\\
  {} & 1   & -2\;& -1\;& -1\;&\m0\;& 0   & -1\;& 0   & 0   & 0   & 0\\
};
\end{tikzpicture}
\end{equation*}
We can now prove our second main result.
\begin{theorem} \label{theorem-2cyclesP5}
  The cone of effective 2-cycles $\Eff_2(X^5_r)$ is linearly generated for $r \leq 5$. 
\end{theorem}
\begin{proof}
  As before, we compute the classes generating $\Lin_2(X^5_5)^\dual$. To avoid an extremely long list, let us say that a subset $\{v_1,\ldots,v_n\}$ of the full set of generators of $\Lin_2(X^5_5)^\dual$ is {\it maximally incident} if every generator can be written in the form $v=v_i+\sum_j a_j F_j + \sum_k b_k (F_k+G_k)$ for some positive integers $a_j$ and $b_k$. Using Lemma \ref{lemma-maximality}, it is sufficient to show that all generators in a maximally incident set are nef. A maximally incident set of generators for $\Lin_2(X^5_5)^\dual$ is as follows:
\begin{equation*}
\def\m{\phantom{-}}%
\begin{tikzpicture}
\matrix[table] {%
  {}          & H^2 & F_1 & F_2 & F_3 & F_4 & F_5 & G_1 & G_2 & G_3 & G_4 & G_5\\
  \alpha      & 1   & -2  &\m0\;&\m0\;&\m0\;&\m0\;& -1\;&\m0\;&\m0\;&\m0\;&\m0\;\\
  \beta       & 1   & -1  & -1\;& -1\;& -1\;& -1\;&\m0\;&\m0\;&\m0\;&\m0\;&\m0\;\\
  \gamma      & 2   & -2  & -2\;& -1\;& -1\;& -1\;& -1\;& -1\;&\m0\;&\m0\;&\m0\;\\
  \delta      & 3   & -3  & -3\;& -3\;& -2\;& -2\;& -1\;& -1\;& -1\;&\m0\;&\m0\;\\
  \varepsilon & 4   & -4  & -4\;& -4\;& -4\;& -4\;& -1\;& -1\;& -1\;& -1\;& -1\;\\
};
\end{tikzpicture}
\end{equation*}
Let us prove that each of these classes is nef:
\begin{itemize}
\item $\alpha$: this class is pulled back from a class $\widetilde{\alpha}$ on the toric variety $X^5_1$. Since the effective cones of toric varieties are linearly generated, $\widetilde{\alpha}$ is nef, hence so too is $\alpha$. 
  
\item $\beta$: this is the class of a codimension-2 linear space touching all 5 lines. We proved that this class is nef in Theorem \ref{theorem-planesinP5}.
\item $\gamma$: let $H$ denote the proper transform of a 4-dimensional linear space containing $L_1$ and $L_2$. We can write the class $\gamma$ as $q+F_1+F_2$, where $q$ is the pushforward of a class in $H \iso X^4_{2,3}$. By Lemma \ref{lemma-maximality} any subvariety intersecting $\gamma$ negatively must intersect $q$ negatively, but by Lemma \ref{prop-appendix2} we can see that $q$ is nef in $H$, so any such subvariety must be contained in $H$. However, again by Lemma \ref{prop-appendix2} the cone of 2-cycles on $H$ is linearly generated, so $\gamma$ has positive degree on any subvariety contained in $H$. 
\item $\delta$: we can prove this is nef by considering the following decomposition into classes of lower degrees.
\begin{equation*}
\def\m{\phantom{-}}%
\begin{tikzpicture}
\matrix[table] {%
  \delta  & H^2 & F_1 & F_2 & F_3 & F_4 & F_5 & G_1 & G_2 & G_3 & G_4 & G_5\\
  \lambda & 1   & -2  &\m0\;&\m0\;& -1  & -1  & -1\;&\m0\;&\m0\;& 0   & 0\\
  q       & 2   & -1  & -3\;& -3\;& -1  & -1  &\m0\;& -1\;& -1\;& 0   & 0\\
};
\end{tikzpicture}
\end{equation*}
Let $H_{23}$ denote a 4-dimensional linear subspace containing the lines $L_2$ and $L_3$. Then, $q$ is the class of the proper transform a quadric threefold in $H_{23}$ containing $L_2$ and $L_3$ and the 3 points of intersection of the other lines with $H_{23}$. As for our proof above for $\gamma$, the proper transform of $H_{23}$ is the fourfold $X^4_{2,3}$, and the class of a quadric containing all 3 points and 2 lines is nef on this space. So any surface class intersecting $q$ negatively must be contained in $X^4_{2,3}$ and hence must be linearly generated.

We must now show the same for $\lambda$. This class is represented by a codimension-2 linear space containing the line $L_1$ and intersecting the lines $L_4$ and $L_5$. Let $H_{14}$ denote a 4-dimensional linear space containing the lines $L_1$ and $L_4$: then, $\lambda$ is represented by any hyperplane inside $H_{14}$ that contains $L_1$ and the point $p_5$ of intersection of $L_5$ with $H_{14}$. Now let $S$ be a irreducible surface in $X^5_5$. If $S$ is contained in some linear space $H_{14}$, then again by Lemma \ref{prop-appendix2}, $S$ is linearly generated. If not, then for each choice of $H_{14}$, we have that the intersection $S \cap H_{14}$ is a curve $C$. If $S \cdot \lambda <0$, then $C$ must be contained in the base locus of the family of hyperplanes containing $L_1$ and $p_5$, which is exactly the plane $P$ spanned by $L_1$ and $p_5$. As we vary the hyperplane $H_{14}$, the corresponding curves $C$ will sweep out the whole surface $S$, and therefore $S$ is contained in the union of all the planes $P$, which is exactly the span of $L_1$ and $L_5$. Again, this shows that $S$ is linearly generated. 
  
\item $\varepsilon$: this class can be written as $D^2$, where $D=2H-\sum_{i=1}^5 E_i$ is the class of the proper transform of a quadric containing all the lines. Since the intersection of nef divisors is nef by Lemma \ref{lemma-nefintersections}, it is enough to prove that $D$ is nef. By semicontinuity, it is enough to show that $D$ is nef for a specific set of 5 disjoint lines.  Note that it is clear that $D$ restricts to an ample divisor on each exceptional divisor $E_i$, so it is enough to check that it has non-negative degree on proper transforms of curves in $\PP^5$.

  Choosing 5 general lines $L_1,\ldots,L_5$ in $\PP^5$ and using Macaulay2 to calculate the base locus $\operatorname{Bs}(L)$ of the linear system $L$ of quadrics containing all 5, we find that $Bs(L)$ is exactly the union of the $L_i$. So for any curve $C$ on $X^5_5$ that comes from $\PP^5$, there is a representative of $D$ meeting the curve properly, and therefore $D \cdot C$ is non-negative as required.
\end{itemize}
\end{proof}
\subsection*{3-cycles on $X^5_r$}
As a  complement to the previous result, we next show that for 3-cycles on blowups of $\PP^5$, linear generation fails as soon as we blow up 4 lines. This is in keeping with the results of \cite{CLO} which show that as we blow up more, linear generation fails sooner for cones of higher-dimensional cycles.

For this result, recall that the {\it Segre cubic 3-fold} is a copy of $\PP^1 \times \PP^2$ embedded in $\PP^5$ by sections of $O(1,1)$.
\begin{proposition} \label{prop-3cyclesonP5}
The cone of effective 3-cycles $\Eff_3(X^5_r)$ is not linearly generated for $r \geq 4$.
\end{proposition}
\begin{proof}
It suffices to prove the claim for $r=4$. For 4 general lines $L_i$ in $\PP^5$ there is a Segre cubic $S$ containing the lines as rulings $\PP^1 \times \{\text{point}\}$. The normal bundle of $L_i$ in $S$ is easily shown to be $O \oplus O$. Fulton's blowup formula \cite[Theorem 6.7]{Fulton} then shows that the proper transform of $S$ on $X^5_4$ has class
\begin{align*}
[\widetilde{S}]  &= 3H^2 - \sum_{i=1}^4 (4F_i + G_i).
\end{align*}
It is straightforward to check that $[\widetilde{S}]$ is not in the linear cone $\Lin_3(X^5_4)$. 
\end{proof}

\section{\texorpdfstring{Curves and divisors on $X^n_r$}{Curves and divisors on X\^{}n\_r}} \label{section-curves}
In this section, we round out the picture for cycles on the varieties $X^n_r$ by considering linear generation of cones of curves and divisors. We write $l$ to denote the pullback of the class of a line in $\PP^n$ and $l_i$ for the class of a line in an exceptional divisor which is contracted by blowing down.
\begin{proposition} \label{prop-curvesinP4}
For $r \leq 7 $ lines in $\PP^4$, the cone of curves $\Eff_1(X^4_r)$ is linearly generated. For $r \geq 10$ lines in $\PP^4$, this cone is not linearly generated.
\end{proposition}
\begin{proof}
For any 3 lines in $\PP^4$, there is a line intersecting all 3. Therefore, the linear cone $\Lin_1(X^4_r)$ is generated by classes $l_i$ for $i=1,\ldots,r$ and classes $l-l_i-l_j-l_k$ for distinct $1\leq i, j ,k \leq r$. The dual cone $\Lin_1(X^4_r)^\dual$ is spanned by $H$, classes $H-E_i$ for $i=1,\ldots,r$ and the class $3H-E_1-\cdots-E_r$.

We claim that the last class is nef for any $r \leq 7$. It suffices to prove this for $r=7$. By semicontinuity, it suffices to prove this for any chosen set of 7 disjoint lines in $\PP^4$. A computation in Macaulay2 shows that, for a set of 7 randomly chosen lines, the base locus of $3H-\sum_{i=1}^7 E_i$ has no component that is a proper transform of a curve in $\PP^4$. On the other hand, the cones of curves of exceptional divisors are linearly generated, and so $3H-\sum_{i=1}^7 E_i$ has non-negative degree on any curve contained in an exceptional divisor. So this class is nef. 

In the other direction, using the intersection numbers in Section \ref{subsection-int} we compute that the top self-intersection number of the divisor $3H-\sum_{i=1}^r E_i$ on $X^4_r$ is $81-9r$. For any $r \geq 10$, this is negative, so the class is not nef,  and therefore $\Lin_1(X^4_r)$ does not equal $\Eff_1(X^5_r)$.
\end{proof}

For 8 lines in $\PP^4$, the base locus of the corresponding class $3H-\sum_{i=1}^8 E_i$ has a component that comes from a curve $C$ of degree 19 in $\PP^4$. Computation shows that $C$ intersects each of the blown-up lines transversely in 6 points; if $C$ were irreducible, we would be able to conclude that $3H-\sum_{i=1}^8 E_i$ is nef and hence that the cone of curves is again linearly generated in this case. Unfortunately, it seems to be out of reach of computation to decide whether $C$ is irreducible. 

\begin{proposition} \label{prop-curvesinP5}
The cone of curves $\Eff_1(X^5_r)$ is linearly generated if and only if $r \leq 5$.
\end{proposition}
\begin{proof}
  In this case, the linear cone $\Lin_1(X^5_r)$ is generated by the $l_i$ together with classes $l-l_i-l_j$. The dual cone $\Lin_1(X^5_r)^\dual$ is then spanned by $H$, classes $H-E_i$ and the class $2H-E_1-\cdots-E_r$.

  In the proof of Theorem \ref{theorem-2cyclesP5}, we showed that $2H-E_1-\cdots-E_5$ is a nef divisor class on $X^5_5$, and therefore $2H-E_1-\cdots-E_r$ is nef on $X^5_r$ for any $r \leq 5$.

In the other direction, the top self-intersection number of the divisor $2H-\sum_{i=1}^r E_i$ on $X^5_r$ is $32-6r$. For any $r \geq 6$, this is negative, so $2H-\sum_{i=1}^r E_i$ is not nef. Hence, $\Lin_1(X^5_r)$ does not equal $\Eff_1(X^5_r)$.
\end{proof}

\begin{proposition} \label{prop-divisorsonP4}
The cone of divisors $\Eff^1(X^4_r)$ is linearly generated if and only if $r \leq 4$. 
\end{proposition}
\begin{proof}
  It suffices to prove the linear generation claim for $r=4$. The linear cone $\Lin^1(X^4_4)$ is spanned by classes $E_i$ and $H-E_i-E_j$, so as in Proposition \ref{prop-curvesinP5}, the dual cone $\Lin^1(X^4_4)^\dual$ is spanned by curve classes $l$, $l-l_i$ and $2l-\sum_{i=1}^4 l_i$. Curves in the classes $l$ and $l-l_i$ evidently sweep out dense open subsets of $\PP^4$, and consequently, they are nef.  For the last class, we argue as follows. For any point $p \in \PP^4$, Schubert calculus shows that there is a plane $\Pi$ touching our 4 blown-up lines $L_i$ and passing through $p$. There is a conic in $\Pi$ passing through the points $L_i \cap \Pi$ and $p$. The proper transform of this conic on $X^4_4$ then has class $2l-\sum_{i=1}^4 l_i$. Since these conics sweep out a dense open subset of $X^4_4$, the class is  nef as required.

Now we will prove that the cone is not linearly generated for $r=5$; again this implies the claim for $r \geq 5$. In this case, the dual of the cone of linear divisors has an extremal ray spanned by the effective class $\gamma=2l-\sum_{i=1}^5 l_i$. We claim that $\gamma$ is not nef. To see this, it suffices to find a big divisor $D$ on $X^4_5$ with $D \cdot \gamma=0$; applying Kodaira's lemma, we can write $D \equiv A+E$ with $A$ ample and $E$ effective, so we must have $E \cdot \gamma <0$. Choose $D=-K=5H-2\sum_{i=1}^5 E_i$. This class has top self-intersection $D^4>0$ as one checks again using the intersection numbers in Section \ref{subsection-int}; therefore it is enough to show that $D$ is nef. The divisor $D$ is represented by the union of the proper transforms of the linear spaces $\operatorname{Span}(L_i,L_{i+1})$ (where subscripts should be read modulo 5), and so it is enough to check that the restriction to each of these proper transforms is nef. Note, however, that each proper transform is isomorphic to $X^3_{2,3}$ and the restriction of $D$ to $X^3_{2,3}$ again decomposes into a union of proper transforms of linear spaces, which are now of the form $X^2_{1,1}$ or $X^2_{0,3}$. Both of these surfaces are toric, so it is straightforward to check that the restriction of $D$ to either surface is nef. Hence $D$ is nef as required. 
\end{proof}

\begin{proposition} \label{prop-divisorsonP5}
  The cone of divisors $\Eff^1(X^5_r)$ is linearly generated if and only if $r \leq 3$.
\end{proposition}
\begin{proof}
  For $r \leq 3$, the variety $\Eff^1(X^5_r)$ is toric, so the claim follows from Proposition \ref{prop-toriccones}.

 For the converse, as above, it suffices to prove the claim when $r=4$. The divisor class $3H-2E_1-2E_2-2E_3-E_4$ is not in the linear cone. This class is represented by the proper transform of a cubic 4-fold double along $L_1, \, L_2, \, L_3$ and containing $L_4$. A straightforward dimension count shows that such 4-folds exist for any 4-tuple of lines in $\PP^5$, and therefore $\Eff^1(X^5_r)$ is not linearly generated.
\end{proof}

\section{\texorpdfstring{Appendix: 2-cycles on $X^3_{2,2}$ and $X^4_{3,2}$}{Appendix: 2-cycles on X\^{}3\_\{2,2\} and X\^{}4\_\{3,2\}}}
In this section, we prove linear generation for the cones of effective 2-cycles on the spaces $X^3_{2,2}$ and $X^4_{3,2}$. These linear generation results were used in the proofs of Theorems \ref{theorem-2cyclesP4} and \ref{theorem-2cyclesP5}.

\begin{lemma} \label{lemma-appendix1}
 The cone of effective 2-cycles $\Eff_2(X^3_{2,2})$ is linearly generated.
\end{lemma}
\begin{proof} Writing down all linear classes on $X^3_{2,2}$ and computing the dual, we find that 
that $\Lin_2(X^3_{2,2})^\dual$ is spanned by the classes
\begin{equation*}
\def\m{\phantom{-}}%
\begin{tikzpicture}
\matrix[table] {%
  {}          & H^2 & HE_1 & HE_2 & -E_3^2 & -E_4^2\\
  \alpha      & 1   &\m0\; &\m0\; &\m0\;   &\m0\;\\
  \beta       & 1   & -1\; &\m0\; &\m0\;   &\m0\;\\
  \gamma      & 1   & -1\; & -1\; &\m0\;   &\m0\;\\
  \delta      & 1   &\m0\; &\m0\; & -1\;   &\m0\;\\
  \varepsilon & 2   & -1\; &\m0\; & -1\;   & -1\;\\
  \kappa      & 2   & -1\; & -1\; & -1\;   & -1\;\\
};
\end{tikzpicture}
\end{equation*}
In each case, irreducible curves representing the class cover a dense open set in $X^3_{2,2}$. For example, the class $\kappa$ is represented by proper transforms of conics touching $L_1$ and $L_2$ and passing through $p_3$ and $p_4$. Choosing a general point $p \in \PP^3$, there is a plane $\Pi$ containing $p$, $p_3$ and $p_4$; this plane intersects $L_1$ and $L_2$ in points $q_1$ and $q_2$, and there is a irreducible conic in $\Pi$ through the 5 points $q_1$, $q_2$, $p_3$, $p_4$ and $p$.   
\end{proof}

\begin{lemma} \label{lemma-x331}
  The cone of effective 2-cycles $\Eff(X^3_{3,1})$ is linearly generated.
\end{lemma}
\begin{proof}
  The dual $\Lin_2(X^3_{3,1})^\dual$ of the linear cone of 2-cycles is spanned by the classes
\begin{equation*}
\def\m{\phantom{-}}%
\begin{tikzpicture}
\matrix[table] {%
  {}          & H^2 & HE_1 & -E_2^2 & -E_3^2 & -E_4^2\\
  \alpha      & 1   &\m0\; &\m0\;   &\m0\;   &\m0\;\\
  \beta       & 1   & -1\; &\m0\;   &\m0\;   &\m0\;\\
  \gamma      & 1   &\m0\; & -1\;   &\m0\;   &\m0\;\\
  \delta      & 2   & -1\; & -1\;   & -1\;   &\m0\;\\
  \varepsilon & 3   & -2\; & -1\;   & -1\;   & -1\;\\
};
\end{tikzpicture}
\end{equation*}
Curves representing the first three classes evidently cover $X$, hence are nef. For the class $\delta$, picking any point $p$ on $L_1$, the plane spanned by $p_2$, $p_3$ and $p$ is covered by irreducible conics with class $\delta$; varying $p$ along $L_1$ these conics cover $X$, and so $\varepsilon$ is nef. Finally, we can write $\varepsilon$ as $\delta+(H^2-HE_1+E_4^2)$; since $\delta$ is nef, any divisor which is negative on $\varepsilon$ must be negative on $H^2-HE1+E_4^2$, which is the class of a line passing through $p_4$ and intersecting $L_1$. If $\pi$ is the plane spanned by $p_4$ and $L_1$, these lines sweep out $\pi$, and therefore $H^2-HE_1+E_4^2$ and hence $\varepsilon$ can be negative only on the proper transform of $\pi$. Note however that $\pi$ is a linear class, and $\varepsilon$ is in the dual of the linear cone, so $\varepsilon$ is in fact nef.
\end{proof}

\begin{lemma} \label{prop-appendix2}
 The cone of effective 2-cycles  $\Eff(X^4_{2,3})$ is linearly generated.
\end{lemma}
\begin{proof}
  The strategy of proof is very similar to previous cases. The dual $\Lin_2(X^4_{2,3})^\dual$ of the linear cone of 2-cycles is spanned by the classes
\begin{equation*}
\def\m{\phantom{-}}%
\begin{tikzpicture}
\matrix[table] {%
  {}          & H^2 & HE_1 & HE_2 & -E_1^2 & -E_2^2 & -E_3^2 & -E_4^2 & -E_5^2\\
  \alpha_1    & 1   &\m0\; &\m0\; &\m0\;   &\m0\;   &\m0\;   &\m0\;   &\m0\;\\
  \alpha_2    & 1   & -1\; &\m0\; &\m0\;   &\m0\;   &\m0\;   &\m0\;   &\m0\;\\
  \alpha_3    & 1   & -1\; & -1\; &\m0\;   &\m0\;   &\m0\;   &\m0\;   &\m0\;\\
  \alpha_4    & 1   &\m0\; &\m0\; & -1\;   &\m0\;   &\m0\;   &\m0\;   &\m0\;\\
  \alpha_5    & 1   & -2\; &\m0\; & -1\;   &\m0\;   &\m0\;   &\m0\;   &\m0\;\\
  \alpha_6    & 2   & -2\; & -1\; & -1\;   & -1\;   &\m0\;   &\m0\;   &\m0\;\\
  \alpha_7    & 2   & -1\; & -1\; &\m0\;   &\m0\;   & -1\;   & -1\;   &\m0\;\\
  \alpha_8    & 3   & -3\; & -2\; & -1\;   &\m0\;   & -1\;   & -1\;   & -1\;\\
  \alpha_9    & 3   & -3\; & -1\; & -3\;   & -1\;   & -1\;   &\m0\;   &\m0\;\\
  \alpha_{10} & 4   & -4\; & -4\; & -1\;   & -1\;   & -1\;   & -1\;   & -1\;\\
};
\end{tikzpicture}
\end{equation*}
The first 6 classes are pulled back from classes on toric varieties that are easily checked to be nef. Similarly, $\alpha_7$ is pulled back from a nef class on $X^4_{2,2}$. The last class $\alpha_{10}$ can be written as $D^2$, where $D$ is the divisor class $2H-\sum_i E_i$. This is the pullback of the class $\widetilde{D}=2H-\sum_i E_i$ on $X^5_5$, which was shown to be nef in Theorem \ref{theorem-2cyclesP5}, so $D$ is a nef divisor, and hence by Lemma \ref{lemma-nefintersections}, we know that $\alpha_{10}=D^2$ is nef too.

 It remains to treat $\alpha_8$ and $\alpha_9$, which we do by decomposition. We start with $\alpha_8$, which can be decomposed as follows:
\begin{equation*}
\def\m{\phantom{-}}%
\begin{tikzpicture}
\matrix[table] {%
  \alpha_8 & H^2 & HE_1 & HE_2 & -E_1^2 & -E_2^2 & -E_3^2 & -E_4^2 & -E_5^2\\
  \beta_1  & 1   &\m0\; & -1   &\m0\;   & 0      &\m0\;   &\m0\;   & -1\;\\
  \beta_2  & 2   & -3\; & -1   & -1\;   & 0      & -1\;   & -1\;   &\m0\;\\
};
\end{tikzpicture}
\end{equation*} 
The class $\beta_1$ is pulled back from a nef class on the toric variety $X^4_{1,1}$ and so is nef. The class $\beta_2$ is represented by the proper transform of a quadric containing $L_1$, intersecting $L_2$, and passing through $p_3$ and $p_4$. Let $\Pi$ be a 3-dimensional linear space containing $L_1$ and passing through $p_3$ and $p_4$; then $\Pi$ intersects $L_2$ in a point, call it $p_2$. Let $\pi$ be the plane spanned by $L_1$ and $p_2$ and let $\pi^\prime$ be any plane containing $p_3$ and $p_4$: then $Q=\pi \cup \pi^\prime$ is a quadric with class $\beta$. Swapping the roles of $p_2$ and $p_3$, say, we see that the base locus of the linear system $|Q|$ consists of a union of lines in $\Pi$. Writing down the classes of linear 1-cycles on $X^3_{3,1}$, we check that $Q$ has positive degree on any such class. So $Q$ is nef inside the proper transform of $\Pi$. It follows that any 2-cycle which intersects $\alpha_8$ and hence $\beta_2$ negatively must be contained in the proper transform of $\Pi$. On the other hand, by Lemma \ref{lemma-x331}, we know that 2-cycles in $X^3_{3,1}$ are linearly generated, and $\alpha_8$ is in the dual of the linear cone. Hence $\alpha_8$ is nef, as required.

For $\alpha_9$ we consider the following decomposition:
\begin{equation*}
\def\m{\phantom{-}}%
\begin{tikzpicture}
\matrix[table] {%
  \alpha_9 & H^2 & HE_1 & HE_2 & -E_1^2 & -E_2^2 & -E_3^2 & -E_4^2 & -E_5^2\\
  \gamma_1 & 1   &\m0\; &\m0\; &\m0\;   &\m0\;   & -1\;   & 0 & 0\\
  \gamma_2 & 1   & -2\; & -1\; & -1\;   &\m0\;   &\m0\;   & 0 & 0\\
  \gamma_3 & 1   & -1\; & -2\; &\m0\;   & -1\;   &\m0\;   & 0 & 0\\
};
\end{tikzpicture}
\end{equation*} 
Again the first class is pulled back from a nef class on a toric variety, hence is nef. For $\gamma_2$ (and similarly for $\gamma_3$), we argue as follows: $\gamma_2$ is represented by the proper transform of a plane $\pi$ containing $L_1$ and intersecting $L_2$ in a point. Let $\Pi$ be the 3-dimensional space spanned by $L_1$ and $L_2$. The proper transform $\widetilde{\Pi}$ of $\Pi$ is a toric variety $X^3_2$. One checks that the proper transform of $\pi$ is a nef divisor in $X^3_2$, hence any surface intersecting $\gamma_2$ negatively must lie in $\widetilde{\Pi}$. On the other hand, as $X^3_2$ is toric, its cone of effective divisors is linearly generated, and so $\gamma_2$ has non-negative degree on surfaces contained in $X^3_2$. 
 
\end{proof}

\end{document}